\documentclass[12pt]{article}

\usepackage{latexsym,amssymb}
\usepackage[T1]{fontenc}
\usepackage[dvips]{graphicx}
\newcommand{\z}{\mathbb Z}
\newcommand{\q}{\mathbb Q}
\newtheorem{lem}{Lemma}[section]
\newtheorem{defn}[lem]{Definition}
\newtheorem{ex}[lem]{Example}
\newtheorem{co}[lem]{Corollary}
\newtheorem{thm}[lem]{Theorem}
\newtheorem{prop}[lem]{Proposition}

\newtheorem{qu}[lem]{Question}
\newenvironment{proof}{\textbf{Proof.}}{\newline\hspace*{\fill}{$\Box$}}

\begin{document}
\title{Large Groups of Deficiency 1}
\author{J.\,O.\,Button\\
Selwyn College\\
University of Cambridge\\
Cambridge CB3 9DQ\\
U.K.\\
\texttt{jb128@dpmms.cam.ac.uk}}
\date{}
\maketitle
\newpage
\begin{abstract}
We prove that if a group possesses a deficiency 1 presentation 
where one of the relators is a commutator then it is the integers 
times the integers, is large, or is as far as possible from being 
residually finite. Then we use this to show that a mapping torus 
of an endomorphism of a finitely generated free group is large if 
it contains the integers times the integers as a subgroup of 
infinite index, as well as showing that such a group is large if 
it contains a Baumslag-Solitar group of infinite index and has a
finite index subgroup with first Betti number at least 2. We give 
applications to free by cyclic groups, 1 relator groups and 
residually finite groups.

\end{abstract}

\section{Introduction}

Recall \cite{p1} that a finitely generated group $G$ is large 
if it has a finite index
subgroup possessing a homomorphism onto a non-abelian free group. 
This is a strong property and implies that
$G$ contains a non-abelian free subgroup \cite{pmn}, 
$G$ is SQ-universal \cite{p1} (every
countable group is a subgroup of a quotient of $G$), $G$ has finite index
subgroups with arbitrarily large first Betti number \cite{lub}, 
$G$ has uniformly exponential word growth \cite{har},
as well as having subgroup growth of strict type 
$n^n$ (which is the largest possible growth for finitely generated groups) 
\cite{luse}, and
the word problem for $G$ is solvable strongly generically in 
linear time \cite{kmss}. Thus 
on proving that $G$ is large we obtain all these other properties for free.

There have been a range of results that give criteria for finitely
generated or finitely presented groups to be large. Starting
with B.\,Baumslag and S.\,J.\,Pride \cite{bp} which showed that groups with a 
presentation of deficiency at least 2 are large, we then have in \cite{gr}
a condition that implies this result, as well as a proof that a group with
a deficiency 1 presentation in which one of the relators is a proper power
is large. This latter result was also independently derived by St\"ohr in
\cite{st} and was followed by conditions for a group with a 
deficiency 0 presentation where some of the relators are proper powers to be
large, due to Edjvet in \cite{edj}. Then further conditions for a finitely
presented group to be large, all of which imply the Baumslag-Pride result,
are by Howie in \cite{how}, G.\,Baumslag in \cite{bag} 
and a characterisation by 
Lackenby in \cite{lac}. In Section 2 we give a criterion, based on the Howie
result, for a finitely presented group $G$ to be large which is purely in
terms of the Alexander polynomial of $G$ and is straightforward to use in
practice. This result is particularly powerful in the case of deficiency 1
groups which are then our focus for much of the rest of the paper.
Of course unlike groups of deficiency 2 or higher, not all groups of deficiency
1 are large: think of $\z$ or the soluble Baumslag-Solitar groups given by
the presentations $\langle x,y|xyx^{-1}=y^m\rangle$ for $m\in\z\backslash
\{0\}$. Other examples of non-large deficiency 1 groups were given by Pride 
and Edjvet in \cite{ep2} consisting of those Baumslag-Solitar groups
$\langle x,y|xy^lx^{-1}=y^m\rangle$ for $l,m\neq 0$ where $l$ and $m$ are
coprime, as well some HNN extensions of these, and one can find the odd further
example in the literature. 

As for large groups of deficiency 1, we have already mentioned those with a
relator that is a proper power and we again have examples in \cite{ep2} with
Theorem 6 stating that the group $\langle x,y|x^ny^lx^{-n}=y^m\rangle$ for
$l,m,n\neq 0$ is large if $|n|>1$ or if $l$ and $m$ are not coprime. Further
results of a more technical nature which give largeness for some other
2 generator 1 relator presentations are in \cite{ephd}. 
At this point it seems difficult to say convincingly either way whether groups
of deficiency 1 are generally large. In this paper we hope to offer substantial
evidence that largeness is a natural property to expect in a deficiency 1
group. Although we will display a few new groups of deficiency 1 which are not
large in Example 3.5(ii), our main results are on establishing families of 
deficiency 1 groups which are all large. In Section 3
we introduce the concept of a non abelian residually abelianised (NARA)
group and this has a number of equivalent definitions,
one of which is that it is finitely generated and non-abelian but has no
non-abelian finite quotients; the idea being that a NARA group $G$ is as
far from being residually finite as possible because we cannot distinguish
$G$ from its abelianisation $G/G'$ by just looking at finite index
subgroups. 
We obtain Theorem 3.6 which states that if $G$ has a deficiency 1 presentation
in which one of its relators is a commutator then $G=\z\times\z$ or $G$ 
is NARA with abelianisation $\z\times\z$ or
$G$ is large.

In \cite{ephd} from 1984 it is asked if those groups
which are an extension of a finitely generated non-abelian free group by
$\z$ are large. They are certainly torsion free groups with a natural
deficiency 1 presentation and are also called
mapping tori of finitely generated non-abelian free group
automorphisms. These groups appear
to make up a sizeable class of deficiency 1 groups but we can expand this
class considerably
by allowing arbitrary endomorphisms in place of automorphisms to obtain
groups which are ascending HNN extensions of finitely
generated free groups. Such groups have
been the attention of much recent research where significant progress has been
made. In particular these groups have been shown to be coherent (every finitely
generated subgroup is finitely presented) in \cite{fh}, Hopfian in \cite{gmsw}
and even residually finite in \cite{bs}. If largeness were added to 
this list (on removing the obvious small exceptions) 
then it would show that such an HNN extension, indeed even a group which is 
virtually such an HNN extension, has all the
nice properties that one could reasonably hope for.

In Section 4 we apply our results to show that for $G$ a mapping torus
of a finitely generated free group endomorphism, we have $G$ is large if it
contains a $\z\times\z$ subgroup of infinite index. Also $G$ is large if it
contains a Baumslag-Solitar subgroup and has a finite index subgroup $H$ 
($\neq\z\times\z$) with $\beta_1(H)\geq 2$. Of course if $\beta_1(H)=1$ for
all $H$ then we would have an example of such a $G$ which is not large.
However we know of no examples apart from the soluble Baumslag-Solitar groups
themselves, and it seems believable that no other $G$ has this property.
 
In Section 5 we restrict to groups $G$ of the form $F$-by-$\z$ where $F$ is
free. By Section 4 $G$ is large if it
contains $\z\times\z$ and $F=F_n$ is of finite rank $n\geq 2$. 
It is known by \cite{bf}, \cite{bfad} and \cite{brink} that
these are exactly the groups of the form $F_n$-by-$\z$ which are not word
hyperbolic. We also show that if $F$ is of infinite rank but $G$ is finitely
generated then $G$ is in fact large.
By combining these results with known facts
about word hyperbolic groups, this allows us to prove in Theorem 5.4 that
if $G$ is any finitely generated group which is virtually free-by-$\z$ then
(apart from the obvious small exceptions) $G$ is SQ-universal, has uniformly
exponential growth and has a word problem that is solvable strongly
generically in linear time. This is also true for the finitely generated
subgroups of $G$.

Section 6 looks at 1 relator groups $G$, where we need only consider the
case where $G$ has a 2 generator 1 relator presentation. We know that by
Section 3 we obtain largeness unless $G=\z\times\z$ (which is easily detected
in the 1 relator case) 
or $G$ is NARA. It is true that 2 generator 1 relator groups which are 
NARA exist, but if we insist that the relator is a product of 
commutators then no examples are known; indeed it was only recently that
non residually finite examples of such groups were given in \cite{nypb}
Problem (OR7). Moreover if the
relator is a single commutator then no examples are known that fail even to
be residually finite (this is Open Problem (OR8) in \cite{nypb})
so in this case not being NARA and hence
large seems very likely. We give methods that show this in practice for
a given presentation. Finally in Section 7 we make the straightforward
but useful observation that a group $G$ is large if and only if the
quotient of $G$ by its finite residual is large, suggesting that the
best setting in which to examine largeness is the residually finite
case. We prove that a residually finite group with infinitely many ends
is large (this is most definitely not true if residual finiteness is
removed) and examine finitely presented groups which are LERF, which is a
strengthening of being residually finite. 

The author would like to
acknowledge helpful comments from Ilya\\
Kapovich, Gilbert Levitt and Alec
Mason, as well as thanking the referee for a thorough reading of the paper.

\section{A Condition for Largeness}

We quote the following facts that we will need about the Alexander polynomial
of a finitely presented group; see \cite{lic}.
Let $G$ be given by a 
finite presentation $\langle x_1,\ldots ,x_n|$
$r_1,\ldots ,r_m\rangle$ and let $G'$ be the derived (commutator) subgroup 
of $G$. Then the abelianisation 
$\overline{G}=G/G'$ is a finitely generated abelian group $\z^{\beta_1(G)}
\times 
T$ for $T$ the torsion subgroup whereas the free abelianisation $ab(G)$ is
$\z^{\beta_1(G)}$. On taking any 
surjective homomorphism $\chi:G\rightarrow\z$, we have the 
Alexander polynomial $\Delta_{G,\chi}\in\mathbb Z[t^{\pm 1}]$ which is a
Laurent polynomial up to the ambiguity of multiplication by the units 
$\pm t^k$ for $k\in\z$. It is defined in the following way: on taking $t$
to be an element in $G$ with $\chi(t)=1$ we have that $t$ acts by
conjugation on $H_1(\mbox{ker }\chi;\z)$, so $H_1(\mbox{ker }\chi;\z)$ is a
module over the group ring $\z[t^{\pm 1}]$ of the integers. It is easy to see
that this is a finitely presented module, for instance we could use the
Reidemeister-Schreier rewriting process to obtain a presentation of 
$\mbox{ker }\chi$ from that of $G$ and then abelianise, which would result
here in an $(n-1)\times m$ presentation matrix. Thus we have the first
elementary ideal which is generated by the maximal minors, these being the
determinants of the matrices left over when we cross off the correct number
of columns to make the resulting matrix square (here we are assuming there
are at least as many columns as rows, or else we let the first elementary
ideal and the Alexander polynomial be zero). Note that this ideal is
independent of the particular presentation matrix chosen for 
$H_1(\mbox{ker }\chi;\z)$.
The definition of the Alexander polynomial $\Delta_{G,\chi}(t)$ is then the
generator (up to units) of the smallest principal ideal containing the first
elementary ideal, or equivalently the highest common factor of the maximal
minors.

The next point is the crucial fact which allows us to use the Alexander
polynomial to detect largeness.
\begin{thm}
If $G$ is a finitely presented group which has a homomorphism $\chi$ onto $\z$
such that $\Delta_{G,\chi}=0$ then $G$ is large.
\end{thm}
\begin{proof}
We have seen that $H_1(\mbox{ker }\chi;\z)$ is a finitely presented module
over $\z[t^{\pm 1}]$ but we can also take rational coefficients and use the
fact that\\ 
$H_1(\mbox{ker }\chi;\z)\otimes_\z\q=H_1(\mbox{ker }\chi;\q)$ is
a finitely presented module over $\q[t^{\pm 1}]$ where $t$ acts in the
same way, and we even have the same presentation matrix. Thus we can define
the Alexander polynomial over $\q$ exactly as above in terms of the first
elementary ideal, and it will be the same polynomial as for $\z$, except
that now it is only defined up to units of $\q[t^{\pm 1}]$ which are now
$qt^{\pm n}$ for $q\in \q\backslash\{0\}$. 
However note that $\Delta_{G,\chi}$ is zero over
$\z$ if and only if it is zero over $\q$. The advantage of moving to rational
coefficients is that $\q[t^{\pm 1}]$ is a principal ideal
domain, so by the structure theorem it is a direct sum of cyclic modules.
Thus the presentation matrix $P$ can be put into canonical form in which all
off-diagonal entries are zero and the diagonal entries are $d_1,\ldots ,d_k$
for $d_i\in\q[t^{\pm 1}]$. By evaluating the first elementary ideal we see
that the Alexander polynomial over $\q$ is $d_1\ldots d_k$ and this is zero
if and only if some $d_i$ is zero which happens if and only if
$H_1(\mbox{ker }\chi;\mathbb Q)$ has a free $\mathbb Q[\z]$-module of
at least rank 1 in its decomposition.

Now we invoke Howie's condition for largeness in \cite{how}
Section 2. Adopting that notation, we let $K$ be the standard connected
2-complex obtained from our finite presentation of $G$,
with $N=\mbox{ker }\chi$ and $\overline{K}$ the
2-complex which is the regular covering of $K$ corresponding to $N$ so that
$\pi_1(\overline{K})=N$. Let $\mathbb F$ be a field: on following
through the proof of \cite{how} Proposition 2.1, we see that if 
$H_1(\overline{K};\mathbb F)$ contains a free $\mathbb F[\z]$-module of rank
at least 1 then the conclusion of the proposition holds. But this is the
hypothesis of \cite{how} Theorem 2.2 which proves that for any sufficiently
large $n$ the finite index subgroup $NG^n$ admits a homomorphism onto
the free group of rank 2.

In our case we have on setting $\mathbb F=\q$
that $H_1(\overline{K};\q)=H_1(\mbox{ker }\chi;\q)$ so if $\Delta_{G,\chi}=0$
we conclude that $G$ is large.
\end{proof}

Note: the above also works if we take $\mathbb F$ to be 
$\mathbb Z/p\mathbb Z$ and $\Delta_{G,\chi}$ vanishes over this field.
\begin{co} If $G$ is a finitely presented group possessing a homomorphism to 
$\z$ with kernel having infinite rational first Betti number then $G$ is large.
\end{co}
\begin{proof}
We have by definition that
$\beta_1(\mbox{ker}\chi;\q)$ is the dimension of\\ 
$H_1(\mbox{ker }\chi;
\q)$ as a vector space over $\q$. It is also the degree of the Alexander
polynomial $\Delta_{G,\chi}$ (where
the degree of a Laurent polynomial in $t$ is the degree of the highest non-zero
power of $t$ minus the degree of the lowest) by \cite{lic} Theorem 6.17 or
\cite{mcm} Section 4. In particular
$\Delta_{G,\chi}=0$ if and only if $\beta_1(\mbox{ker }\chi;\q)$ is infinite, 
so this claim now follows directly from Theorem 2.1.
\end{proof}

Note: The Corollary is most definitely not true for all finitely generated 
groups; we do require a finite number of relators too, as can be seen by
the example of the restricted wreath product $\z\wr\z$. 

\section{Deficiency 1 groups}

The deficiency of a finite presentation is the number of generators minus the
number of relators and the deficiency $def(G)$ of a finitely presented group
$G$ is the maximum deficiency over all presentations. (It is bounded above
by $\beta_1(G)$ so is finite.) We know that groups of deficiency
at least 2 are large so it seems reasonable to ask whether we can use our 
criterion to obtain large groups with lower deficiencies, for instance 
deficiency 1. In fact this case turns out
to be a very fruitful choice, both from the point of view that calculating the
Alexander polynomial of a deficiency 1 group is more efficient than for lower
deficiencies, and because of the behaviour of deficiency in finite covers. 
Given a presentation for a group $G$ with $n$ generators and $m$ relators and
an index $i$ subgroup $H$ of $G$,
we can use Reidemeister-Schreier rewriting to obtain a presentation for 
$H$ of $G$ with $(n-1)i+1$ generators and $mi$ relators,
thus the deficiency of $H$ is at least $(def(G)-1)i+1$. So if
$def(G)=1$ then either $def(H)=1$ for all $H\leq_f G$ or $H$, and thus $G$,
is large anyway by \cite{bp}.

\begin{thm}
If $G$ is a group with a deficiency 1 presentation $\langle x_1,\ldots ,x_n|$
$r_1,\ldots ,r_{n-1}\rangle$ where one of the relators is of the form
$x_ix_jx_i^{-1}x_j^{-1}$ then $G$ is large if the subgroup of $ab(G)$ generated
by the images of $x_i$ and $x_j$ has infinite index.
\end{thm}
\begin{proof}
Without loss of generality we can reorder the generators and so we can
assume we have the relator $x_1x_2x_1^{-1}x_2^{-1}$. As $ab(G)$ is a free
abelian group $\z^{\beta_1(G)}$ of finite rank, we have that $x_1$ and $x_2$
generate a free abelian subgroup of strictly smaller rank. Therefore there
must exist a surjective homomorphism $\chi:G\rightarrow\z$ with $x_1$ and
$x_2$ in the kernel, as well as coprime integers $k_3,\ldots ,k_n$ such
that $k_3\chi(x_3)+\ldots +k_n\chi(x_n)=1$. Therefore there exists a matrix
$M\in GL(n-2,\z)$ such that its first column is $(k_3,\ldots ,k_n)$ and
this gives rise to an automorphism $\beta$ of $\z^{n-2}$ sending the standard
basis $e_3,\ldots ,e_n$ (where we think of $e_i$ as the image of the generator
$x_i$ in the abelianisation $\z^n$ of the free group $F_n$) to a new basis
$b_3,\ldots ,b_n$. Now by \cite{ls} I.4.4, we have an automorphism $\alpha$
of $F_n$ that fixes $x_1,x_2$ and induces $\beta$ on $e_3,\ldots ,e_n$. On
rewriting our presentation in terms of $y_1=\alpha(x_1),\ldots
,y_n=\alpha(x_n)$, we now have $\chi(y_3)=1$ and so we can regard 
$H_1(\mbox{ker }\chi;\z)$ as a $\z[t^{\pm 1}]$ module where $t$ is equal to
$y_3$ and acts by conjugation.
We can obtain a presentation matrix $P$ for this module by performing
Reidemeister-Schreier rewriting on $G$ using $\{t^j:j\in\z\}$ as a Schreier
transversal. We find that our original relation $x_1x_2x_1^{-1}x_2^{-1}$
becomes the set of group relations $x_{1,j}x_{2,j}x_{1,j}^{-1}x_{2,j}^{-1}$
where $x_{1,j}=t^jx_1t^{-j}$ and $x_{2,j}=t^jx_2t^{-j}$. To obtain the
equivalent relation for $P$, we abelianise and regard each of these group 
relations as the same module relation multiplied by powers of $t$. But this 
becomes zero, thus giving us a zero column in $P$.

The crucial point about the group presentation having deficiency one is that
this makes $P$ a square matrix (of size $n-1$). This means that the Alexander
polynomial $\Delta_{G,\chi}$ is merely the determinant of $P$, which must be
zero owing to the zero column, hence we have largeness of $G$ by Theorem 2.1.
\end{proof}

\begin{co}
If $G=\langle x_1,\ldots , x_n|r_1,\dots ,r_{n-1}\rangle$ has a deficiency 1
presentation with a relator $r_1=x_1x_2x_1^{-1}x_2^{-1}$ and the abelianisation
$\overline{G}=\z\times\z\times\z/m\z$ for $m\geq 2$ then $G$ is large.
\end{co}
\begin{proof}
We are done by Theorem 3.1 unless the images $\overline{x_1},\overline{x_2}$ 
in $\overline{G}$ 
generate a finite index subgroup $S$ of $\overline{G}$, but if so then $S$
must have $\z$-rank equal to that of $\overline{G}$, which is 2. However $S$
is generated by two elements so in this case $S$ can 
only be isomorphic to $\z\times\z$. Now take a homomorphism $\theta$ from 
$G$ onto
$\z/j\z$ for some $j\geq 2$ such that $S$ is in the kernel. We require another
generator $g\in\{x_3,\ldots ,x_n\}$ such that $\theta(g)$ generates 
Im $\theta$ but this can be achieved by taking an appropriate automorphism
of the free group of rank $n$ that fixes $x_1$ and $x_2$, just as in Theorem
3.1. We now perform 
Reidemeister-Schreier rewriting to obtain from our original presentation of $G$
a deficiency 1 presentation for ker $\theta$ consisting of $nj+1$ generators
and $nj$ relators. We have $g^i$, $0\leq i<j$ as a Schreier transversal 
for ker $\theta$ in $G$ and on setting $x_{1,i}=g^ix_1g^{-i}$ and
$x_{2,i}=g^ix_2g^{-i}$, which will all be amongst the
generators for our presentation of ker $\theta$ given by this process (because
$x_1,x_2\in S\leq\mbox{ker }\theta$), 
our original relator $r_1$ gives rise to  
$j$ relators $x_{1,i}x_{2,i}x_{1,i}^{-1}x_{2,i}^{-1}$ in the presentation
for our subgroup. As these disappear when we abelianise, we see that 
$\beta_1(\mbox{ker }\theta)$ is at least $j+1$ and we are done by Theorem 3.1.
\end{proof}

It might be felt that requiring two generators to commute in a deficiency 1
presentation is rather restrictive but most of
the rest of our results are based on
finding deficiency 1 groups $G$ which have a finite index subgroup $H$
possessing such a presentation. This means $\beta_1(H)\geq 2$ and Corollary
3.2 will apply unless the abelianisation $\overline{H}=\z\times\z$. We now
discuss a generalisation of the property of being residually finite which 
allows us to avoid this exception.

Recall that a group $G$ is residually finite if the intersection $R_G$ over all
the finite index subgroups $F\leq_f G$ is the trivial group $I$. Although this
works perfectly well as a general definition, it is most useful when $G$ is
finitely generated and that will be our assumption here. Our
motivation for the next definition is to ask: how badly can a group fail to be
residually finite and what is the worst possible case? The first answer that
would come to mind is when $G$ ($\neq I$) has no proper finite index subgroups
at all, but we are dealing with groups possessing positive first Betti
number and hence infinitely many subgroups of finite index. By noting that
elements outside the commutator subgroup $G'$ cannot be in $R_G$, we obtain
our condition.
\begin{defn}
We say that the finitely generated group $G$ is {\bf residually abelianised} if
\[G'=\bigcap_{F\leq_f G}F.\]
If further $G$ is non-abelian then we say it is {\bf NARA (non-abelian 
residually abelianised)}.
\end{defn}
Note that by excluding $G$ being abelian, we have that $G$ residually finite
implies $G$ is not NARA. The definition has many equivalent forms but the 
general idea is that a NARA group cannot be distinguished 
from its abelianisation if one only uses standard information about its finite
index subgroups.
\begin{prop}
Let $G$ be finitely generated and non-abelian with commutator subgroup $G'$,
abelianisation $\overline{G}=G/G'$ and let $R_G$ be the intersection of the
finite index subgroups of $G$. The following are equivalent:\\
(i) $G$ is NARA.\\
(ii) $G$ has no non-abelian finite quotient.\\
(iii) $G$ has no non-abelian residually finite quotient.\\
(iv) If $a_n(G)$ denotes the number of finite index subgroups of $G$ having
index $n$ then $a_n(G)=a_n(\overline{G})$ for all $n$.\\
(v) For all $F\leq_f G$ we have $F'=G'$.\\
(vi) For all $F\leq_f G$ we have $F\cap G'=G'$.\\
(vii) For all $F\leq_f G$ we have $F'=F\cap G'$.\\
\end{prop}
\begin{proof}
The equivalence of (i) with (ii) is immediate on dropping
down to a finite index normal subgroup. We have (iii) implies (ii) and (i)
implies (iii) as any residually finite image of $G$ must factor through
$G/R_G$. As for (iv), this is just using the
index preserving correspondence between the subgroups of $\overline{G}$ and 
the subgroups of $G$ containing $G'$.

As for the rest, we have that $F'\leq F\cap G'\leq G'$ whenever $F$ is a 
subgroup of $G$. If (i) holds for $G$ with $F$ a finite index subgroup then
$R_F=R_G=G'$ but $R_F$ is inside $F'$ so $F'$ and $G'$ are equal, giving (v). 
This
immediately implies (vi) and (vii) so we just require that these two in turn 
imply
(i). This is obvious for (vi) and for (vii) we can adopt the proof of
\cite{lr} Theorem 4.0.8 which states that if $\Gamma$ is a residually finite
group then for each of its (nontrivial) cyclic subgroups there exists a
homomorphism onto another (nontrivial) cyclic group which can be extended
to a finite index subgroup of $\Gamma$.
If (i) fails then take $F\leq_f G$ and $g\in G'$ but $g\notin F$.
Dropping down to $N\leq F$ with
$N\unlhd_fG$, we have $H=N\langle g\rangle\leq_f G$ and
$H/N\cong\langle g\rangle/(N\cap\langle g\rangle)$. Thus $g\notin H'$ because
by being outside $N$ it survives under a homomorphism from $H$ to an abelian
group. But $g$ is certainly in $H\cap G'$.
\end{proof}

The importance of condition (vii) holding for $G$ 
is that we fail to pick up extra abelianisation in finite covers 
$F\leq_f G$ since $F/F'$ is just $F/(F\cap G')\cong FG'/G'\leq_fG/G'$. 
In particular $\beta_1(F)=\beta_1(G)$ so $G$ is not large.

\begin{ex}
\end{ex}
(i) The Thompson group $T$ is NARA. This group has a 2 generator
2 relator presentation with abelianisation $\z\times\z$ and its commutator
subgroup $T'$ has no proper finite index subgroups as $T'$ is infinite and
simple; see \cite{cfp}.
But for $F\leq_fT$ we have $F\cap T'\leq_fT'$ thus $T'\leq F$.\\
(ii) If $G$ is infinite but has no proper finite index
subgroups then $G$ is NARA. Moreover for any such $G$ and any
residually abelianised group $A$ we have $\Gamma=G*A$ is 
NARA because if $N\unlhd_f\Gamma$ then $N\cap G\unlhd_f G$ so $G\leq N$.
This implies that the normal closure $C$ of $G$ is in $N$ so $\Gamma/N$ must
be abelian as it is a finite quotient of $\Gamma/C\cong A$. 
A famous example that will do for $G$ is the Higman group $H$
with 4 generators and 4 relators as introduced in \cite{hig}. It has 
$\beta_1(G)=0$ so its deficiency must be zero. Thus $H*H$ is NARA so 
it too has no proper finite index subgroups, since it is 
infinite and equals its own commutator subgroup. By repeating this construction
we obtain $H*\ldots *H$ using $n$ copies of $H$ which gives us examples of
NARA groups $G_n$ which can have arbitrarily many generators (by the 
Grushko-Neumann theorem) and with $\beta_1(G)=0$ and deficiency zero.
In order to obtain examples of deficiency 1 NARA groups we can take the
free product of $G_n$ with $\z$ or $\z\times\z$ so that the resulting groups
need arbitrarily many generators and have their first Betti number equal
to 1 or 2. Of course there are no groups of deficiency two or higher which
are NARA because they are all large.\\
(iii) There are 1 relator groups which are NARA: the first example
dates back to a short paper \cite{ba} of G.\,Baumslag in 1969 entitled 
``A non-cyclic one-relator group all of whose finite quotients are cyclic'' 
with the group in question being
\[\langle a,b|a=a^{-1}b^{-1}a^{-1}bab^{-1}ab\rangle.\]
However Baumslag-Solitar groups are not NARA, as can be seen by taking
quotients onto dihedral groups.

In terms of its wide application, the following is our main result on largeness
of deficiency 1 groups.
\begin{thm}
If $G$ has a deficiency 1 presentation $\langle F_n|R\rangle$ where one of the
relators is a commutator in $F_n$ then exactly one of these occurs:\\
(i) $G=\z\times\z$.\\
(ii) $G$ is NARA with abelianisation $\z\times\z$.\\
(iii) $G$ is large.\\
In particular if there exists $H\leq_f G$ with
$\overline{H}\neq\z\times\z$ then $G$ is large.
\end{thm}
\begin{proof} If our relator $r=uvUV$ for $u,v$ words in the generators for
$F_n$ then we can regard $r$ as the commutator of two generators simply by
adding $u$ and $v$ to the generators and their definitions to the relators,
noting that this does not change the deficiency.
We must have $\beta_1(G)\geq 2$ and if $\beta_1(G)\geq 3$ (or if the subgroup
generated by the images of $u$ and $v$ has $\z$-rank less than 2) then $G$ is
large by Theorem 3.1. Moreover if $\overline{G}$ has non-trivial torsion then
we have largeness by Corollary 3.2 so the only case in which the given
presentation for $G$ does not show largeness is when $\overline{G}=\z 
\times\z$ with the subgroup $S=\langle \overline{u},\overline{v}\rangle$
(where $\overline{u}$ and $\overline{v}$ are the images of $u$ and $v$ in
$\overline{G}$) being of finite index. If neither (i) nor (ii) are true then
Proposition 3.4 tells us that none of the given seven conditions hold for $G$, 
so the failure of (vii) means that there is a subgroup $L\leq_f G$
with $\gamma$ in $L\cap G'$ but not in $L'$. Consequently $\overline{L}=L/L'$
is an abelian group which surjects to $L/(L\cap G')$ with $\gamma$ in the 
kernel. But $L/(L\cap G')$ is isomorphic to $LG'/G'$ which is a finite index
subgroup of $\overline{G}$ and thus is equal to $\z\times\z$. As this is a
Hopfian group, we conclude that $\beta_1(L)\geq 2$ but $\overline{L}\neq\z
\times\z$. We know that $L$ also has a deficiency 1 presentation which we can
obtain from $G$ by rewriting and if one of the relations in such a presentation
for $L$ were a commutator then we could conclude by Theorem 3.1 or Corollary
3.2 applied to this presentation that $L$, and hence $G$, is large. In fact
although we cannot guarantee this, we now show that there is a finite index
subgroup of $L$ which has such a presentation, along with the necessary
abelianisation. We do this by keeping track of what happens to our original
relator $r$ when we perform Reidemeister-Schreier rewriting
on dropping to a finite index subgroup. The process of rewriting for a subgroup
$H$ in $G=\langle g_1,\ldots ,g_n\rangle$ involves taking a Schreier
transversal $T$ to obtain a generating set for $H$ of the form $tg_i(\overline{
tg_i})^{-1}$, where $t\in T$ and $\overline{x}$ is the element of $T$ in the
same coset as $x$. In particular by taking $t$ equal to the identity we have
that if a generator $g_i$ of $G$ is in $H$ then it becomes a generator of $H$.
Moreover as the relators for $H$ are obtained by expressing the relators 
$tr_jt^{-1}$ in terms of these generators, any relator made up solely from 
those $g_i$ which are contained in $H$  will remain unchanged
in the presentation for $H$.

First we take $H$ to be the normal
subgroup of finite index in $G$ which is the inverse image of $\langle
\overline{u},\overline{v}\rangle$ under the abelianisation map from $G$ to
$\overline{G}$. Then the deficiency 1 presentation for $H$ has $r$ as one
of its relators thus we have largeness for $H$ and $G$ unless
$\overline{H}=\z\times\z$. However if so then we must have  $H'=H\cap G'$. 
This is because otherwise we
have a surjective but non-injective homomorphism from $H/H'$ to $H/(H\cap G')
\cong\z\times\z$, thus we can use Hopficity again. Moreover as $G'\leq H$ we 
get $H'=G'$.

Let $k,l$ be the minimum positive integers such that $a=u^k$ and $b=v^l$
are in $L$. We set $N=\langle H',a,b\rangle$ which is a finite index normal
subgroup of $H$ and of $G$, with $G/N$ abelian. We rewrite for $N$ in $H$ 
in two stages; first we drop to the subgroup
with exponent sum of $u$ equal to 0 mod $k$ and rewrite using the
transversal $u^i, 0\leq i<k$, and then we do the same with $v$. In both of
these stages $a$ and $b$ will be amongst the generators for
$N$ and our relator $uvUV$ in $H$ gives rise to a relator $ava^{-1}v^{-1}$ 
after the first rewrite, and then this becomes $abAB$. Thus we have $G'\leq N
\unlhd_f G$ with $N$ having a deficiency 1 presentation which includes 
generators $a,b$ and the relator $abAB$. 

Finally we go from $N$ to the subgroup $L\cap N\leq_f G$ which on rewriting
will keep $a$ and $b$ because they are generators in the presentation for $N$
which also lie in $L\cap N$, and  consequently $abAB$ remains too. Now our
$\gamma\in L$ from before which is in $G'\backslash L'$ is also in $N$ as 
$G'\leq N$. But from above we have a surjective homomorphism from $L$ to 
$\z\times\z\times\z/j\z$ for some $j \geq 2$ with $\gamma$ mapping onto the 
$\z/j\z$
factor. But then we can restrict this surjection to the finite index subgroup
$L\cap N$ which also
contains $\gamma$ so $L\cap N$ has the right presentation and the right
homology to obtain largeness.
\end{proof}

Note that Example 3.5 (ii) shows that case (ii) in Theorem 3.6 can occur but 
this is the only type of example known to us.

\section{Ascending HNN extensions of free groups}

A wide and important class of deficiency 1 groups is obtained by taking a free
group $F_n$ with free basis $x_1,\ldots ,x_n$ and an endomorphism $\theta$
of $F_n$ to create the mapping torus 
\[G=\langle x_1,\ldots ,x_n,t|tx_1t^{-1}=\theta(x_1),\ldots ,
tx_nt^{-1}=\theta(x_n)\rangle.\]
We call such a presentation a standard presentation for $G$.
We do not assume that $\theta$ is injective or surjective. However
there is a neat way of sidestepping the non-injective case using \cite{kpnt}
where it is noted that $G$ is isomorphic to a mapping torus of an injective
free group homomorphism $\tilde{\theta}:F_m\rightarrow F_m$ where $m\leq n$.
Of course it might be that $F_n$ is non-abelian but $m=0$ or 1, however this
would mean that 
$G=\z$ or $\langle a,t|tat^{-1}=a^k\rangle$ for $k\neq 0$. However in these
cases $G$ is soluble and so is definitely not large. Therefore we will assume
throughout that $\theta$ is injective, whereupon $G$ is also called an
ascending HNN extension of the free group $F_n$, where we conjugate the base 
$F_n$ to an isomorphic subgroup of itself using the stable
letter $t$.

We have that our base $F_n=\langle x_1,\ldots ,x_n\rangle$ embeds in
$G$ and we will refer to this copy of $F_n$ in $G$ as $\Gamma$. Then $t\Gamma
t^{-1}=\theta(\Gamma)$ which is equal to $\Gamma$ if and only if $\theta$ is
surjective in which case $G$ is free by $\z$. Otherwise $\theta
(\Gamma)<\Gamma$, with $\theta(\Gamma)$ being isomorphic to $F_n$ meaning that 
it has infinite index in $\Gamma$.

Once an ascending HNN extension $G$ is formed, there is an obvious homomorphism
$\chi$ from $G$ onto $\z$ associated with it which is given by $\chi(t)=1$
and $\chi(\Gamma)=0$, so that
\[\mbox{ker }\chi=\bigcup_{i=0}^\infty t^{-i}\Gamma t^i.\]
In the case of an automorphism ker $\chi$ is just
$\Gamma$ but otherwise $\theta(\Gamma)=t\Gamma t^{-1}<\Gamma$ so that ker 
$\chi$ is a strictly ascending union of free groups, thus is infinitely 
generated and locally free, but never free because 
$\beta_1(\mbox{ker }\chi;\q)\leq\beta_1(\Gamma;\q)$.

The following result, which is Lemma 3.1 in \cite{gmsw}, allows us to recognise
ascending HNN extensions ``internally''.
\begin{lem}
A group $G$ with subgroup $\Gamma$ is an ascending HNN extension with $\Gamma$
as base if and only if there exists $t\in G$ with\\
(1) $G=\langle \Gamma,t\rangle$;\\
(2) $t^k\not\in\Gamma$ for any $k\neq 0$;\\
(3) $t\Gamma t^{-1}\leq\Gamma$.
\end{lem}

A strong property that ascending HNN extensions of free groups possess is that
they are coherent by \cite{fh}, in fact the result is more general and gives
us this description which is Theorem 1.2 and Proposition 2.3 in \cite{fh}.
\begin{thm}
If $G=\langle t,F\rangle$ is an ascending HNN extension of the (possibly
infinitely generated) free group $F$ with associated homomorphism $\chi$
and $H$ is a finitely generated subgroup
of $G$ then $H$ has a finite presentation of the form
\[\langle s,a_1,\ldots ,a_k,b_1,\ldots ,b_l|sa_1s^{-1}=w_1,\ldots ,
sa_ks^{-1}=w_k\rangle\]
where $a_i,b_j\in\mbox{ker }\chi$ and $k,l\geq 0$, with
$w_1,\ldots ,w_k$ words in the $a_i$ and the $b_j$.
\end{thm}

The next proposition gives us standard but useful properties of ascending
HNN extensions.

\begin{prop}
Let $G$ be an ascending HNN extension
\[\langle x_1,\ldots ,x_n,t|tx_1t^{-1}=\theta(x_1),\ldots ,
tx_nt^{-1}=\theta(x_n)\rangle\]
with respect to the injective endomorphism $\theta$ of the finitely generated
free group $\Gamma=F_n$ with free basis $x_1,\ldots ,x_n$
and let $\chi$ be the associated homomorphism.\\
(i) Each element $g$ of $G$ has an expression of the form 
$g=t^{-p}\gamma t^q$ for $p,q\geq 0$ and $\gamma\in\Gamma$.\\
(ii) For each $j\in\mathbb N$ we have for $s=t^j$ the normal subgroup 
$G_j=\langle \Gamma,s\rangle$ of index $j$ in $G$ with $G/G_j\cong\z/j\z$
which has presentation
\[\langle x_1,\ldots ,x_n,s|sx_1s^{-1}=\theta^j(x_1), \ldots ,sx_ns^{-1}=
\theta^j(x_n)\rangle.\]
(iii) If $H\leq_f G$ then $H$ is also an ascending HNN extension of a
finitely generated free group with respect to the
(restriction to $H$ of the) same associated homomorphism $\chi$.\\
(iv) If $\Delta\leq_f\Gamma$ then $H=\langle\Delta,t\rangle$ has finite index 
in $G=\langle\Gamma,t\rangle$.
\end{prop}
\begin{proof}
(i) is \cite{fh} Lemma 2.2 (1).\\ 
(ii) This is \cite{kp} Lemma 2.2 (1).\\
(iii) This can be proved directly but we may use the fact that $H$ has a
presentation as in Theorem 4.2. We now show that $G$ and all its finite index 
subgroups have deficiency exactly 1 so we must have $l=0$ in this
presentation and then the result follows. This also demonstrates that in
proving certain ascending HNN extensions of free groups are large, we are
genuinely finding new examples as opposed to groups that could be proved
large by the Baumslag-Pride result \cite{bp}.

We know that $H$ has a deficiency 1 presentation by Reidemeister-Schreier
rewriting the standard presentation for $G$. By \cite{swll} Proposition
3.6 (ii) we have that the 2-complex $C$ associated to this standard
presentation of $G$ is aspherical and the Euler characteristic $\chi(C)=
1-(n+1)+n=0$. Therefore the finite cover $\tilde{C}$ of $C$ with fundamental
group $H$ has $\chi(\tilde{C})=0$ and is also aspherical. Now by the Hopf
formula we have that an upper bound for the deficiency of any finitely
presented group $\Gamma$ is $\beta_1(\Gamma)-\beta_2(\Gamma)$. As $\tilde{C}$
is a $K(H,1)$ space we have $\beta_1(H)-\beta_2(H)=\beta_1(\tilde{C})-\beta_2
(\tilde{C})=1-\chi(\tilde{C})=1$ and so our deficiency 1 presentation for $H$
is best possible.\\ 
(iv) Let $\gamma_1,\ldots ,\gamma_d$ be a transversal for $\Delta$ in $\Gamma$.
If we can show that $H\cap K\leq_f K$ for $K$ the kernel of the associated
homomorphism then we are done, despite the fact that
$K$ and $H\cap K$ are infinitely generated, because $t\in H$ and any $g\in G$
is of the form $kt^m$ for $k\in K$.

The set
\[S=\{t^{-m}\gamma_it^m:m\in\mathbb N,1\leq i\leq d\}\]
contains an element of every coset of $H\cap K$ in $K$. This can be seen by
writing $k\in K$ as $t^{-m}\gamma t^m$ for $\gamma\in\Gamma$ using (i). Then
there is $\gamma_i$ such that $\gamma\gamma_i=\delta\in\Delta$. This means that
$kt^{-m}\gamma_it^m=t^{-m}\delta t^m$ which is in $H$ and in $K$. We now show
that the index of $H\cap K$ in $K$ is at most $d$. Note that for $q>p$,
any element of the form $t^{-p}\gamma_it^p$ is in the same coset as some 
element of the form $t^{-q}\gamma_jt^q$ because $\theta^{q-p}(\gamma_i)
\gamma_j^{-1}=\delta$ for some $\delta\in\Delta$ 
and some $j\in\{1,\ldots ,d\}$,
thus giving $t^{-p}\gamma_it^p(t^{-q}\gamma_jt^q)^{-1}=t^{-q}\delta t^q$ which
is in $H\cap K$. Therefore we proceed as follows: $S$ is a set indexed by
$(l,i)\in\mathbb N\times\z/d\z$ and we refer to $l$ as the level. Choose a 
transversal $T$ for $H\cap K$ in $K$ from $S$ which a priori could be
infinite and let $g_1$ be the element in
$T$ with smallest level $l_1$ (and smallest $i$ if necessary). Then for each
level $l_1+1,l_1+2,\ldots$ above $l_1$ there is an element in $S$ with this
level that is in the same coset of $H\cap K$ as $g_1$ and so cannot be in $T$.
Cross these elements off from $S$ and now take the next element $g_2$ in $T$
according to our ordering of $S$. Certainly $g_2$ with level $l_2$ has not been
crossed off and we repeat the process of removing one element in each level 
above $l_2$; as these are in the same coset as $g_2$ they too have not been 
erased
already. Now note that we can go no further than $g_d$ because then we will
have crossed off all elements from all levels above $l_d$; thus we must have a
transversal for $H\cap K$ in $K$ of no more than $d$ elements.
\end{proof}

Let $G=\langle
F_n,t\rangle$ be the mapping torus of an injective endomorphism $\theta$ of
the free group $F_n$. We say that $\theta$ has a periodic conjugacy class 
if there exists $i>0$, $k\in\z$ and $w\in F_n\backslash
\{1\}$ such that $\theta^i(w)$ is
conjugate to $w^k$ in $F_n$. If this is so with $\theta^i(w)=vw^kv^{-1}$ then
let us take the endomorphism $\phi$ of $F_n$ such that $\phi=\iota_v^{-1}
\theta^i$ where we use $\iota_v$ to denote the inner automorphism of $F_n$
that is conjugation by $v$.
We have on setting $\Delta=\langle w\rangle$ and $s=v^{-1}t^i$ 
that the subgroup $\langle\Delta,s\rangle$ of $G$ is an ascending HNN
extension with base $\Delta$ and stable letter $s$ by Lemma 4.1. Consequently
it has the presentation $\langle s,w|sws^{-1}=w^k\rangle$. 
These presentations are part of the
famous family of 2 generator 1 relator subgroups known as the Baumslag-Solitar
groups. We define the Baumslag-Solitar group $BS(j,k)=\langle x,y|xy^jx^{-1}=
y^k\rangle$ for $j,k\neq 0$ (and without loss of generality $j>0$). 
We have that $G$ contains $BS(1,k)$ for some $k$ if and only 
if $G$ has a periodic conjugacy class where $\theta^i(w)$ is conjugate 
to $w^k$. Furthermore if there exists 
$i,j>0$, $k\in\z$ and $w\in F_n\backslash\{ 1\}$ with
$\theta^i(w^j)$ conjugate in $F_n$ to $w^k$ then $k=dj$ and $\theta^i(w)$ is
conjugate to $w^d$ so that $\theta$  has a periodic conjugacy class.
Indeed $G$ cannot contain a subgroup isomorphic to a Baumslag-Solitar group
$BS(j,k)$ unless $j=1$ (or $j=k$ in which case $G$ contains $BS(1,1)$ anyway).

We can now deal with ascending HNN extensions of free groups which contain
$\z\times\z$.

\begin{thm}
If $\theta$ is an injective endomorphism of the free group $\Gamma$ of rank
$n$ with $w\in F_n\backslash\{1\}$ such that $\theta(w)=w$ then there is a
finite index subgroup $\Delta$ of $\Gamma$ and $j\geq 1$ such that $\Delta$
has a free basis including $w$, and such that $\theta^j(\Delta)\leq\Delta$.
\end{thm}
\begin{proof}
We use the classic result \cite{hjnr} of Marshall Hall Jnr. that if $L$ is a 
non-trivial
finitely generated subgroup of the non-abelian free group $F_n$ then there
is a finite index subgroup $F$ of $F_n$ such that $L$ is a free factor of 
$F$. We just need to put $L=\langle w\rangle$ so that $F=\langle w\rangle
*C$ for some $C\leq F_n$ with $w$ a basis element for $F$. The second condition
is the crucial part. The aim is to repeatedly pull back $F$; although we do
not have $F\leq\theta^{-1}(F)$ in general as this is equivalent to $\theta(F)
\leq F$ which would mean we are done, we do find that the index is 
non-increasing. To see this note that $\theta^{-1}(F)=\theta^{-1}(F\cap\theta
(\Gamma))$ and $\theta^{-1}\theta(\Gamma)=\Gamma$ as $\theta:\Gamma\rightarrow
\theta(\Gamma)$ is an isomorphism. Now the index of 
$F\cap\theta(\Gamma)$ in $\theta(\Gamma)$ is preserved
by applying $\theta^{-1}$ to both sides, so it is equal to the index of 
$\theta^{-1}(F)$ in $\Gamma$. But the index of $F\cap\theta(\Gamma)$ in 
$\theta(\Gamma)$ is no more than that of $F$ in $\Gamma$, thus
$[\Gamma:\theta^{-i}(F)]$ gives us a non-increasing sequence
which must stabilise at $N$ with value $k$. When it does we have for $i\geq 0$
that $\theta^{-(i+N)}(F)$ is just moving around the finitely many index $k$
subgroups. Although it happens that $\theta^{-1}$ does not in general  
permute these index $k$ subgroups, we must land on some such subgroup
$\Delta$ twice so we have $j\geq 1$ with $\theta^{-j}(\Delta)=\Delta$, giving
$\Delta\geq\theta^j(\Delta)$.

We now show that, although the rank of $\theta^{-i}(F)$ reduces whenever the
index reduces, we can keep $w$ as an element of a free basis each time we pull
back. This time we restrict $\theta$ to an injective homomorphism from 
$\theta^{-1}(F)$ to $F$ with image $\theta\theta^{-1}(F)$. As $\theta
\theta^{-1}(F)$ is a finitely generated subgroup of $F$ containing a free basis
element $w$ of $F$, we can ensure $w$ is in a free basis for $\theta\theta^{-1}
(F)$ (for instance see \cite{ls} Proposition I.3.19). Now $\theta^{-1}(F)$ and 
$\theta\theta^{-1}(F)$ are isomorphic via $\theta$ with inverse $\phi$ say, 
so a basis $b_1,\ldots ,b_r$ for the latter gives rise to a basis $\phi(b_1),
\ldots ,\phi(b_r)$ for $\theta^{-1}(F)$ and if $b_1=w$ then $\phi(b_1)=w$.
\end{proof}
\begin{co}
If $G=\langle\Gamma,t\rangle$ is a mapping torus of an injective endomorphism
$\theta$ of the free group $\Gamma$ of rank $n$ and 
$\z\times\z\leq G$ then we have $H\leq_f G$ such that $H$ has a deficiency 1
presentation $\langle x_1,\ldots ,x_m,s|r_1,\ldots ,r_m\rangle$ including a
relator of the form $sx_1s^{-1}x_1^{-1}$.
\end{co}
\begin{proof}
As $BS(1,1)\leq G$ we have $w\in\Gamma\backslash\{1\}$ with $\theta^i(w)=
vwv^{-1}$ for some $v\in\Gamma$,
thus on dropping to the index $i$ subgroup $H$ of $G$ 
given by $H=\langle\Gamma,t^i\rangle$ and setting $\phi$ to be $\iota_v^{-1}
\theta^i$ where $\iota_v(x)=vxv^{-1}$, we can
assume that there is $w\in\Gamma\backslash\{1\}$ with $\phi(w)=w$ and that
$H$ is an ascending HNN extension of $\Gamma$ via the injective endomorphism
$\phi$ and with stable letter $t_H$ say. So by
Theorem 4.4 we have $\Delta\leq_f\Gamma$ with $\Delta$ having a free basis
$w,x_2,\ldots ,x_m$ and $j\geq 1$ with $\phi^j(\Delta)\leq\Delta$. Thus by
Proposition 4.3 (ii) and (iv) we have that $L=\langle\Delta,s|s=t_H^j\rangle$ 
has finite index in $G$ and by Lemma 4.1 $L$ 
is an ascending HNN extension 
with base $\Delta$ and stable letter $s$. Thus on taking the standard 
presentation for $L$ given by conjugation of $s$ on this free basis for 
$\Delta$, we see that it has deficiency 1 
with a relator equal to $sws^{-1}w^{-1}$.
\end{proof}

We can now gain largeness for a range of mapping tori.
\begin{co} If $G=\langle\Gamma,t\rangle$ is the mapping torus of an
endomorphism $\theta$ of the free group $\Gamma$ of rank $n$ and $\z\times\z
\leq G$ then $G=\langle x,y|xyx^{-1}=x^{\pm 1}\rangle$ or is large.
\end{co}
\begin{proof}
We can assume without loss of generality that $\theta$ is injective
because if not then we
can replace $\theta$ with $\tilde{\theta}$ which is
an injective endomorphism of a free group $F_m$ with $m\leq n$, and then $G$
is still equal to the mapping torus of $F_m$ using $\tilde{\theta}$ and this
will contain $\z\times\z$. Hence we are in the case of
Corollary 4.5 which allows us to apply Theorem 3.6 to 
$H\leq_f G$. As $\z\times\z\leq G$, we do not have $m=0$ and only the two
groups above for $m=1$. Otherwise $G$ and hence $H$ contain a
non-abelian free group for $m\geq 2$
so $H$ is not in case (i) of Theorem 3.6. By Proposition
4.3 (iii) $H$ is an injective mapping torus of a finitely generated free
group endomorphism and so the recent result \cite{bs} 
of Borisov and Sapir tells us that $H$ is residually finite, so it is not 
NARA. Thus $H$ and $G$ are large.
\end{proof}

Although it might be said that one only requires the NARA property to apply
Theorem 3.6 and not the full force of residual finiteness, it should be
pointed out that there are mapping tori $G$ of injective endomorphisms of the
free group $F_2$ such that the abelianisation $\overline{G}=\z\times\z$ and
such that for any finite index subgroup $N$ which is normal in $G$ with
$G/N$ soluble, we have $\overline{N}=\z\times\z$.

We finish this section by looking at those mapping tori $G$ of endomorphisms
of free groups which contain an arbitrary Baumslag-Solitar subgroup. Our 
results are not quite definitive because we need $\beta_1(G)\geq 2$ in
order to apply our methods and we cannot show that $G$ necessarily has a
finite index subgroup with that property. However this is the only obstacle
to largeness.
\begin{thm}
If $G=\langle\Gamma,t\rangle$ is a mapping torus of an endomorphism
$\theta$ of the free group $\Gamma$ of rank $n$ which contains a 
Baumslag-Solitar subgroup $BS(j,k)$ then either $G$ is large or $G=BS(1,k)$ 
or $\beta_1(H)=1$ for all $H\leq_f G$.
\end{thm}
\begin{proof}
As usual we assume that $\theta$ is injective. We know that
$G$ can only contain Baumslag-Solitar subgroups of type $BS(1,k)$ or $BS(k,k)$
for $k\neq 0$ and as we have already covered those which contain $BS(1,1)$, we
need only consider $BS(1,k)\leq G$ for $k\neq\pm 1$. If there is some
$H\leq_f G$ with $\beta_1(H)\geq 2$ then we can replace $G$ by $H$ because
$H$ is a mapping torus by Proposition 4.3 (iii) and $BS(1,k)\cap H
\leq_fBS(1,k)$ so $H$ contains some Baumslag-Solitar group too.
Therefore we are looking
at the situation where we have a periodic conjugacy class of the form $w\in F_n
\backslash\{1\}$ and $i>0$ with $\theta^i(w)$ conjugate to $w^d$ for some
$d\neq\pm 1$. Just as in 
the $\z\times\z$ case, we drop down to a finite index subgroup and change our
automorphism by an inner automorphism, so we can assume that
$\theta(w)=w^d$. Now we follow the proof of Theorem 4.4 to get $F\leq_f\Gamma$
with $\langle w\rangle$ a free factor of $F$, observing that $w\in\theta^{-1}
(F)$ so that we keep $w$ as we pull back $F$. Note that we can assume $w$ is
not a proper power by the comment before Theorem 4.4
so we can also preserve $w$ in a 
free basis each time because $w^d\in\theta\theta^{-1}(F)$ and if $w^c\in
\theta\theta^{-1}(F)$ for $0<|c|<|d|$ then the element $u\in\theta^{-1}(F)$
with $\theta(u)=w^c$ cannot be a power of $w$ but $\theta(u^d)=\theta(w^c)$,
hence contradicting injectiveness. Thus $w^d$ can be extended to a free basis
for $\theta\theta^{-1}(F)$ by \cite{ls} Proposition I.3.7
meaning that $w$ will be in the corresponding basis
for $\theta^{-1}(F)$.

We can now work to obtain an equivalent version of Corollary 4.5. Having gone
from $G$ to the finite index subgroup $H$ which is an ascending HNN extension
of $\Gamma$ via the injective homomorphism $\theta$, we see as before that by
repeatedly pulling back $F$ we obtain $\Delta\leq_f\Gamma$ which has a free
basis including $w$ and with $\theta^j(\Delta)\leq\Delta$. Hence the HNN
extension $J$ of $\Delta$ using $\theta^j$ with stable letter $s$ has
finite index in $H$, as well as a deficiency 1 presentation that
includes the relator $sws^{-1}w^{-e}$ for $e\neq 0,\pm 1$. We will also
require later that $e\neq 2$ and this can be obtained by taking the subgroup
of $J$ of index 2 as in Proposition 4.3 (ii) so that now the relator would
be $ sws^{-1}w^{-4}$. Now consider taking a surjective
homomorphism $\chi$ from $J$ to $\z$ (which must send $w$ to 0). 
If $\beta_1(J)$ were 1 then the only available $\chi$ would be the homomorphism
associated to this HNN extension so it would send $s$ to
1 (or $-1$). However if not then we can find $\chi'\neq\chi$ as we
have $\beta_1(H)\geq 2$ and hence $\beta_1(J)\geq 2$ because $J\leq_f H$.
Hence we have a non-trivial homomorphism $\chi'-k\chi$ that sends $s$ to 0
(which can be made surjective by multiplying 
by the right constant) where $k$ is $\chi'(s)$. On
evaluating the Alexander polynomial of $J$ with respect to this homomorphism,
we proceed as in Theorem 3.1 and discover that the group relation
$sws^{-1}w^{-e}$ becomes the module relation $(1-e)w$ when rewritten and
abelianised. Thus we have a column in our square presentation matrix
consisting of all zeros except $1-e$ in the row corresponding to the
generator $w$. Thus if we apply Theorem 2.1 using the field $\z/p\z$ with $p$
a prime dividing $1-e$ then our Alexander polynomial is zero so we have
largeness for $J$, and hence for $G$.
\end{proof}

Although we do not have a proof that a mapping torus of a 
free group endomorphism containing a Baumslag-Solitar subgroup of infinite
index has a finite index subgroup with first Betti number at least two, the
statement of Theorem 4.7 is still useful in a practical sense because if
we are presented with a particular group $G$ of this form that we would
like to prove is large, we can enter the presentation into a computer and
ask for the abelianisation of its low index subgroups. As soon as we see one
with first Betti number at least two, we can conclude largeness. 
Note that in \cite{kp} it is conjectured that a mapping torus of a free
group endomorphism is word hyperbolic if it does not contain Baumslag-Solitar
subgroups, and if this and the above question on having a finite index 
subgroup with first Betti number at least two are both true then we have
proved largeness for all the non word hyperbolic ascending HNN extensions
of finitely generated free groups (with the obvious exception of the soluble
Baumslag-Solitar groups).

We can even say something if $G$ is a mapping torus of an injective
endomorphism of an infinitely generated free group in the case when $G$ is
finitely generated, thanks to the power of \cite{fh} by using Theorem 4.2.
We immediately see that either $G$ has
deficiency at least two and so is large, or $l=0$ in which case $G$ is also
a mapping torus of an endomorphism of a finitely generated free group 
and so the results of this section apply. 

\section{Free by Cyclic Groups}

If in the previous section we use an automorphism $\alpha$ of a free group
$F$ to form our mapping torus, we obtain a semidirect product $F\rtimes
_\alpha\z$ and every free-by-$\z$ group is of this form. We already have
largeness for a range of these groups.
\begin{thm}
If the finitely generated group $G$ is free-by-$\z$ then $G$ is large if the
free group $F$ is infinitely generated, or if $\z\times\z\leq G$ and $F$
has rank at least 2.
\end{thm}
\begin{proof}
If $F$ is infinitely generated then 
applying Theorem 4.2 with $H=G$ tells us that $G$ has deficiency
at least 2. This is because if $l=0$ then the kernel of the associated
homomorphism $\chi$ is 
\[\bigcup_{n=0}^\infty s^{-n}As^n\qquad\mbox{where }A=\langle a_1,\ldots ,a_k
\rangle\]
but then $\beta_1(\mbox{ker }\chi;\q)\leq\beta_1(A;\q)$ whereas we have 
$F=\mbox{ker }\chi$.

If $F$ has finite rank then this is Corollary 4.6. Note that a semidirect
product $A\rtimes B$ is residually finite if both $A$ and $B$ are residually
finite and $A$ is finitely generated, so we do not need to use \cite{bs}
when applying this corollary to $G$.
\end{proof}

In contrast G.\,Baumslag gives in \cite{nncycloc} an example of an infinitely
generated free-by-$\z$ group with every finite quotient cyclic so that
this group is not residually finite. Indeed its finite residual $R_G$ must
contain $G'$
and hence it is not large because every finite index subgroup $F$ of $G$ has
the property that all of the finite quotients of $F$
are abelian, as $F'\leq G'\leq R_G=R_F$.
This is a striking
demonstration of how largeness and residual finiteness are best suited to
finitely generated groups. He then proves in \cite{bgfin} that finitely
generated groups which are $F$-by-$\z$ for $F$ an infinitely generated free
group are residually finite. As for the residual finiteness of finitely
generated groups that are ascending HNN extensions of infinitely generated
free groups, this appears to be open (it corresponds to the case $l>0$ in 
Theorem 4.2); indeed that they are Hopfian is Conjecture 1.4 in \cite{gmsw}.
\begin{co}
If $G$ is $F_n$-by-$\z$ for $F_n$ the free group of rank $n$ then either $G$
is large or $G$ is word hyperbolic or $G=BS(1,\pm 1)$.
\end{co}
\begin{proof}
It is known by \cite{bf}, \cite{bfad} and \cite{brink} that such a $G$ being
word hyperbolic is equivalent to $G$ containing no subgroups isomorphic to 
$\z\times\z$ and also to the automorphism $\alpha$ having no periodic conjugacy
classes.
\end{proof}

In fact the equivalence of the last two notions can be proved directly by
quoting the classical result of Higman \cite{hg} which says that an
automorphism of a free group that maps a finitely generated subgroup into
itself maps it onto itself. So if $\alpha$ has a periodic conjugacy class
we can assume there is $i>0$ and $w\in F_n\backslash\{1\}$ such that
$\theta^i(w)$ is conjugate to $w^{\pm 1}$.

This leaves us with an important question:
\begin{qu}
If $G$ is $F_n$-by-$\z$ for $n\geq 2$ and $G$ is a word hyperbolic group
then is $G$ large?
\end{qu}

As for whether the six consequences of largeness given in the introduction
hold for these groups $G$, the first is obvious whereas it is unknown if
$G$ has superexponential subgroup growth or has infinite virtual first Betti
number: Question 12.16 by Casson in \cite{bpr} is equivalent to asking whether
there exists $H\leq_f G$ with $\beta_1(H)\geq 2$.
However being word
hyperbolic means that the other properties are known to hold, giving us a
definitive result for these three cases.
\begin{thm} If $G$ is finitely generated and is virtually free-by-$\z$ then
for all finitely generated subgroups $H$ of $G$ we have:\\
(i) $H$ has word problem solvable strongly generically in linear time\\
(ii) $H$ has uniformly exponential growth\\
(iii) $H$ is SQ-universal\\
unless $H$ is virtually $S$ for $S=\z$ or $\z\times\z$.
\end{thm}
\begin{proof}
We have shown by Theorem 5.1 and Corollary 5.2 that any free-by-$\z$ group
$G$ which is finitely generated (excepting $BS(1,\pm 1)$ and $\z$) is large
or is non-elementary word hyperbolic. This implies SQ-universality (by
\cite{ol} for the hyperbolic case and \cite{p1} when $G$ is large) and
uniformly exponential growth (by \cite{koubi} for the hyperbolic case and
\cite{har} when we have largeness).
Then Corollary 4.1 in \cite{kmss} (where
we also have the relevant definitions) proves (i) if $G$ has a finite index
subgroup with a non-elementary word hyperbolic quotient. 
Moreover if $G\leq_f\Gamma$ then the properties (i), (ii) and (iii) hold for
$\Gamma$ as well, by \cite{kmss} for (i), \cite{har} for (ii) and \cite{pmn}
for (iii).

Finally if $H\leq G$ where $G$ is free-by-$\z$ then $H/(H\cap F)\cong HF/F
\leq G/F=\z$ so either this is trivial with $H\leq F$ hence $H$ is free, or
it is isomorphic to $\z$ and so $H$ is an extension of the free group
$H\cap F$ by $\z$ 
(although if $F$ is finitely generated then $H\cap F$ is not necessarily
finitely generated if $H$ has infinite index in $F$).
Thus if $H$ is finitely generated then (i), (ii), (iii)
or the exceptions hold for $H$ too, and if $L$ is a finitely generated
subgroup of the virtually free-by-$\z$ group $\Gamma$ with $G\leq_f\Gamma$ 
then $L\cap G\leq_f L$ so $L\cap G$ is a finitely generated 
subgroup of $G$, hence we gain our properties 
for $L\cap G$ and then also for $L$.
\end{proof}

\section{1 Relator Groups}
Groups having a finite presentation with only 1 relator have been much studied.
It is known that such a group contains a non-abelian free group unless it is
isomorphic to $BS(1,m)$ or is cyclic, see \cite{ls} II Proposition 5.27 and
\cite{wil}. Indeed if the 
presentation has at least 3 generators then we know by \cite{bp} that the
group is large so we need only concern ourselves here with 2 generator 
1 relator presentations $\langle a,b|r\rangle$. Largeness is also known by
\cite{gr} and \cite{st} when $r$ is a proper power, which is exactly when
the group has torsion, but for $l,m$ coprime we have by \cite{ep2} that
$BS(l,m)$ is not large as it has virtual first Betti number equal to 1, and
similarly Example 3.5 (iii) is not large. Thus another direction in which to
go when looking for large 2 generator 1 relator groups is if $r$ is in the 
commutator subgroup of $F_2$, as at least that gives first Betti number equal
to 2. The starting case to be considered here should be to take $r$ actually
equal to a commutator and application of Theorem 3.6 gives us a near
definitive result.
\begin{co}
If $G=\langle a,b|uvUV\rangle$ where $u$ and $v$ are any elements of $F_2=
\langle a,b\rangle$ with
$uvUV$ not equal to $abAB$, $baBA$ or their cyclic conjugates when reduced and
cyclically reduced then $G$ is large or is NARA.
\end{co}
\begin{proof}
It is well known that $G=\z\times\z$ if and only if the relator is of the 
above form
(equivalently if and only if $u,v$ form a free basis for $F_2$); see for
instance \cite{kms} Theorem 4.11. Otherwise
we are in Theorem 3.6 case (ii) or (iii).
\end{proof}
\begin{qu} If $G=\langle a,b|uvUV\rangle$ then can $G$ be NARA?
\end{qu}

No examples are known to us. This is an important question because a yes
answer gives us a non residually finite 1 relator
group with the relator a commutator, the existence of
which is Problem (OR8) in the problem list at \cite{nypb} (however there
it is shown that non residually finite
examples exist if the relator is merely in the commutator
subgroup) and a no answer gives us largeness.
We can prove that we do not have NARA groups in a whole range of cases.
\begin{prop}
If $G=\langle F_2|uvUV\rangle$ then $G$ can only be NARA if 
$u,v\notin F_2'$ with the images of $u$ and $v$ generating the
abelianisation $\z\times\z$ of $F_2$ and such that $u$ is a free basis element 
for $F_2$ or $G_u=\langle F_2|u\rangle$ is NARA, along with the
same condition for $v$.
\end{prop}
\begin{proof}
If the images of $u$ and $v$ do not generate the homology of $F_2$ up to finite
index then we are done by Theorem 3.1. Now suppose that $G_u$ is not
NARA or $\z$ (the latter happening if and only if $u$ is an element of a
free basis for $F_2$) then as $G$ surjects to $G_u$ we see that a non-abelian 
finite image of $G_u$ is also an image of $G$. Then swap $u$ and $v$.

Otherwise we can take a free basis $\alpha,\beta$ for $F_2$ such that 
there are $k,l
\geq 1$ with $u$ equivalent to $\alpha^k$ in homology and $v$ to $\beta^l$. 
If $k>1$
then consider the homomorphism of $F_2$ into the linear (hence residually
finite) group $SL(2,\mathbb C)$ given by
\[\alpha\mapsto \left( \begin{array}{cc} e^{\pi i/k} & 0 \\
0 & e^{-\pi i/k} \end{array} \right), \qquad
\beta\mapsto \left( \begin{array}{rr} 1 & 1 \\ 0 & 1\end{array} \right).\]
This is non-abelian but does make $u$ and $v$ commute so we are done by
Proposition 3.4 (iii).
\end{proof}

In fact there are other ways to conclude that $G=\langle F_2|uvUV\rangle$ is 
not NARA and hence large, for instance the powerful algorithm
of K.\,S.\,Brown in \cite{bks}, which determines whether a 2 generator 
1 relator group is a mapping torus of an injective endomorphism of a finitely
generated group (which must necessarily be free), can be used 
(along with \cite{bs} proving that such groups are residually finite). If all
else fails then there is the option of using the computer to find the
abelianisation of some low index subgroups of $G$ and look for one which is not
$\z\times\z$ in order to obtain largeness. For instance in \cite{ephd}
it is shown that 
\[G=\langle a,b|a^{k_1}b^{l_1}a^{k_2}b^{l_2}a^{k_3}b^{l_3}\rangle\]
for $k_1k_2k_3,l_1l_2l_3\neq 0$ and $k_1+k_2+k_3=l_1+l_2+l_3=0$ is large
for all possibilities except for one particular relator 
(on which we use the computer to find a finite index
subgroup $H$ of the form in Theorem 3.6 and with $\beta_1(H)>2$) and
two infinite families of relators (these are of the required form so we can
use Brown's algorithm) thus we have shown 
that the remaining cases are all large.

We finish this section by mentioning a conjecture of P.\,M.\,Neumann in
\cite{pmn} from 1973: that a 1 relator group is either SQ-universal or is
isomorphic to $BS(1,m)$ or cyclic (the next comment that ``a proof of this by
G.\,Sacredote seems to be almost complete now'' turns out with hindsight to 
be somewhat over optimistic). In addition to presentations with 3 or more 
generators or
with the relator a proper power, at least this can be seen to be true for
those 2 generator 1 relator groups $G$ which are free by cyclic, and possibly
ascending HNN extensions of free groups if the two questions at the end of
Section 4 are true, as well as
if $G$ is virtually a group of this type. We make progress on this question 
from a different direction in the next section.

\section{Residual Finiteness}

It appears that often when we have a counterexample to statements about
largeness, this is achieved by taking a group which is not residually
finite. The following straightforward observation suggests why:
\begin{prop}
A group $G$ is large if and only if the residually finite group $G/R_G$ is
large, where $R_G$ is the intersection of all the finite index subgroups
of $G$.
\end{prop}
\begin{proof}
A group is large if any quotient is large, whereas any homomorphism from $G$
to a residually finite group factors through $G/R_G$ and if $H\leq_f G$ then
$R_H=R_G$.
\end{proof}

Thus perhaps we should take the same approach as those who count finite
index subgroups of finitely generated groups by only considering residually
finite groups. However the example of $G=BS(2,3)$, where $G/R_G$ is soluble
but not finitely presented, means we can lose good properties of our original
group.
This assumption removes the obvious counterexamples which are SQ-universal
but not large, for instance taking free products of groups with no finite
index subgroups, and then the two properties begin to look more similar.
A recent result of \cite{os} shows that finitely generated groups with
infinitely many ends are SQ-universal. Whilst they cannot all be large, as
evidenced by these free products, it is straightforward to establish this
in the residually finite case by adapting an argument of Lubotzky from
\cite{lub}.
\begin{thm}
A residually finite group with infinitely many ends is large.
\end{thm}
\begin{proof}
If $\Gamma=G_1*_\phi G_2$ where $\phi$ is an isomorphism from $A$ a finite
subgroup of $G_1$ to $B$ a finite subgroup of $G_2$ then, as $G_1$ will be
residually finite, we can take $M\unlhd_fG_1$ with $M\cap A=I$ and
$[G_1:M]>2|A|$, meaning that the subgroup $AM/M$ of $G_1/M$ has index greater
than 2 and is isomorphic to $A$. We can also get $N\unlhd_fG_2$ with
$N\cap B=I$ and $[G_2:N]>2|B|$. Now we can form $(G_1/M)*_{\overline{\phi}}
(G_2/N)$, 
where $\overline{\phi}(aM)=\phi(a)N$ provides an isomorphism from $AM/M$
to $BN/N$. This is a quotient of $\Gamma$ and is virtually free by \cite{ser}
II Proposition 11,
with the index conditions ensuring that it is virtually non-abelian free
(see \cite{ser} 2.6 Exercise 3).

As for HNN extensions $\Gamma=G*_\phi$, where $\phi$ is an isomorphism with
domain a finite proper subgroup $A$ of $G$, and $\phi(A)\leq G$ is conjugate
to $A$ in $\Gamma$ via the stable letter $t$, we now take $N\unlhd_f\Gamma$
such that there exists $g\in G$ with $ag\notin N$ for all $a\in A$, which
implies that $AN\neq GN$. Thus $AN/N$ and $\phi(A)N/N$ are subgroups of
$GN/N$ which are conjugate in $\Gamma/N$ via $tN$, thus isomorphic, with
both of these proper subgroups. Hence the HNN extension $\langle GN/N,s|
s(aN)s^{-1}=\phi(a)N\rangle$ can be formed and this is a non-ascending
HNN extension of a finite group, thus it is virtually non-abelian free.
\end{proof}

A group $G$ is called LERF (equivalently subgroup separable) if every finitely
generated subgroup is an intersection of finite index subgroups. An
observation in \cite{bn} is that $G$ cannot be LERF if there is a finitely
generated subgroup $H$ and $t\in G$ with $tHt^{-1}\subset H$, thus proper
ascending HNN extensions of finitely generated groups are never LERF. If
$G$ is $F_n$-by-$\z$ and $\beta_1(G)\geq 2$ then it is possible for $G$ to
be simultaneously free-by-cyclic and also to be a proper ascending HNN
extension of a finitely generated free group with respect to another
associated homomorphism onto $\z$, so being LERF is quite a lot stronger
than merely being residually finite. However if we assume this we can gain
some very specific conclusions.
\begin{thm}
If $G$ is finitely presented and LERF then either $G$ has virtual first
Betti number equal to 0, or $G$ is large, or $G$ is virtually $L\rtimes\z$
for $L$ finitely generated.
\end{thm}
\begin{proof}  
Any group with positive first Betti number is an HNN extension (it is a
semidirect product $(\mbox{ker }\chi)\rtimes\z$ for $\chi$ a homomorphism
onto $\z$) but if it is finitely presented then \cite{bs78} tells us that
it is an HNN extension $L*_\phi$ with stable letter $t$ and with
$L$ and the domain $A$ of $\phi$ both
finitely generated. Thus on taking $H\leq_fG$ with $\beta_1(H)>0$ we have
$H=L*_\phi$ with presentation 
\[\langle L,t|ta_it^{-1}=\phi(a_i)\mbox{ for }1\leq i\leq m\rangle\]
where $a_1,\ldots,a_m$ is a generating set for $A$.
Now if $A\neq L$ then, as subgroups of LERF groups are also
LERF, we have $F\leq_fH$ which contains $A$ but not $L$. We can take $N\leq F$
which is normal in $H$ and of finite index. This gives us
$AN\leq F<LN$ and so we can argue
as in Theorem 7.2 to get a non-ascending HNN extension: 
We have that $LN/N$ is finite and $AN\neq LN$ implies that $AN/N$ is a proper
subgroup of $LN/N$.
Now the isomorphism $\phi$ from $A$ to $\phi(A)$ is induced
by conjugation by $t$, so that in $H/N$ we have that $AN/N$ and $\phi(A)N/N$
are conjugate by the element $tN$. Hence the induced map $\overline{\phi}(aN)
=\phi(a)N$ is well defined and is an isomorphism between these two subgroups.
Therefore we can form $(LN/N)*_{\overline
{\phi}}$ with the domain of $\overline{\phi}$ equal to $AN/N$, and this has
presentation
\[\langle LN/N,s|s(a_iN)s^{-1}=\overline{\phi}(a_iN)
\mbox{ for }1\leq i\leq m\rangle\]
which is an image of $H$ under the homomorphism $L\mapsto LN/N,t\mapsto s$
and is large by \cite{ser} II Proposition 11 so $H$ is large too.

Otherwise the HNN extension is ascending but the LERF condition means 
that $H$ is in fact a semidirect product.
\end{proof}

Going back to deficiency 1 groups, we have \cite{hill} Theorem
6 which states that if $G$ has deficiency 1 and is an ascending HNN 
extension with
base the finitely generated subgroup $L$
then the geometric dimension of $G$ (thus the cohomological dimension) is at 
most two. But on combining this with \cite{gmsw} Corollary 2.5, which states 
that if $L$ is of type $FP_2$ and has cohomological dimension 2 then
$G$ has cohomological dimension 3, we see that if $G$ has deficiency 1 then
$L$ finitely presented (or even $FP_2$) implies that $L$ is in fact
free. This means that the only way a deficiency 1 group $G$ could 
fail the 
Tits alternative of not being virtually soluble and not containing $F_2$ is
for $G$ to be an ascending HNN extension $L*_\phi$
where $L$ is finitely generated but not finitely presented
and where $L$ fails the Tits alternative. It is unknown whether this can occur
but at least we can conclude
the Tits alternative holds for coherent deficiency 1 groups.
In \cite{wil} it is shown that a soluble deficiency 1 group is $BS(1,m)$ or
$\z$. As these groups are coherent, this result is easily extended to
virtually soluble deficiency 1 groups by the above and the result in 
\cite{bs78}
that a finitely presented group $G$ with $\beta_1(G)>0$ which does not
contain $F_2$ is an ascending HNN extension of a finitely generated group. 

We can use this to obtain some results about LERF groups in particular cases.
\begin{co}
If $G$ is LERF and has deficiency 1 then either $G$ is SQ-universal or $G$ is
$BS(1,\pm 1)$ or $\z$ or $G=L\rtimes \z$ for $L$ finitely generated
but not finitely presented.
\end{co}
\begin{proof}
Certainly $\beta_1(G)>0$ so by the proof of Theorem 7.3 $G$ is large or 
equals $L\rtimes\z$ with $L$ finitely generated. If $L$ is finitely
presented then the above comment shows that $L$ is free and Theorem 5.4 (iii) 
applies if $L$ is non-abelian free.
\end{proof}

We can finish by making some progress on P.\,M.\,Neumann's conjecture given
in the last section.
\begin{co}
If $G$ is a 1 relator group which is LERF then either $G$ is SQ-universal
or $G$ is cyclic or $G=BS(1,\pm 1)$.
\end{co}
\begin{proof}
We know that $G$ is finite cyclic or has deficiency at least two or has
deficiency one whereupon Corollary 7.4 applies. But if $G=L\rtimes\z$ for
$L$ finitely generated then $L$ must be free by \cite{bks} Section 4.
\end{proof}

\end{document}